\documentclass{article}

\usepackage{authblk}
\usepackage{caption}
\usepackage{hyperref}

\usepackage{color}
\usepackage{amssymb}
\usepackage{amsmath}
\usepackage{amsthm}
\usepackage{arydshln}
\usepackage{bbm}
\usepackage{blindtext}
\blindmathtrue
\usepackage{mathrsfs}
\usepackage{enumerate}
\usepackage{verbatim}
\usepackage{graphicx}
\usepackage{subfig}
\numberwithin{equation}{section}
\providecommand{\keywords}[1]
{
  \small	
  \textbf{\textit{Keywords:}} #1
}

\newcommand{\R}{\mathbb{R}}

\renewcommand{\P}{\mathbb{P}}
\newcommand{\E}{\mathbb{E}}

\newcommand{\cB}{\mathcal{B}}

\newcommand{\cU}{\mathcal{U}}
\newcommand{\cX}{\mathcal{X}}

\newtheorem{Theorem}{Theorem}[section]
\numberwithin{Theorem}{section}
\newtheorem{Proposition}[Theorem]{Proposition}
\newtheorem{Corollary}[Theorem]{Corollary}
\newtheorem{Lemma}[Theorem]{Lemma}
\newtheorem{Remark}[Theorem]{Remark}

\newtheorem{Definition}[Theorem]{Definition}

\title{Wasserstein distance in terms of the comonotonicity Copula}

\author[1]{Mariem Abdellatif}
\author[2]{Peter Kuchling}
\author[3]{Barbara R\"udiger}
\author[4]{Irene Ventura}

\affil[1]{University of Wuppertal, Germany\\ \texttt{abdellatif@uni-wuppertal.de}}
\affil[2]{Bielefeld University of Applied Sciences and Arts, Germany\\ \texttt{peter.kuchling@hsbi.de}}
\affil[3]{University of Wuppertal, Germany\\ \texttt{ruediger@uni-wuppertal.de}}
\affil[4]{University of Wuppertal, Germany\\ \texttt{ventura@uni-wuppertal.de}}

\date{} 

\begin{document}

\maketitle

\begin{abstract}
The aim of this article is to  write the $p$-Wasserstein metric $W_p$  with the $p$-norm, $p\in [1,\infty)$, on $\R^d$ in terms of copula. In particular for the case of one-dimensional distributions,  we get  that the copula employed to get the optimal coupling of the Wasserstein distances is the comotonicity copula. We obtain the equivalent result also for $d$-dimensional distributions  under the  sufficient and necessary condition that these have the same dependence structure of their one-dimensional marginals, i.e that the $d$-dimensional distributions share the same copula. Assuming $p\neq q$, $p,q$ $\in [1,\infty)$ and that the probability measures $\mu$ and $\nu$ are sharing the same copula,  we also analyze the Wasserstein distance $W_{p,q}$ discussed in \cite{Alfonsi}  and  get an upper and lower bounds of  $W_{p,q}$ in terms of $W_p$, written in terms of comonotonicity copula. We show that as a consequence the lower and upper bound of $W_{p,q}$ can be written in terms of  generalized inverse functions.
\end{abstract}


\allowdisplaybreaks

\section{General introduction}
Wasserstein distances or Kantorovich–Rubinstein distances have been introduced first by Leonid Kantorovich in $1939$ \cite{MR129016} and they are used in many areas of pure and applied mathematics. The concept of Wasserstein distance is motivated by the concept of optimal transportation and  is based on finding an appropriate coupling between two marginal probability measures. Indeed, the optimal transport cost between two probability measures $\mu$ and $\nu$ on the product of two Polish probability spaces $\mathcal{X}\times\mathcal{Y}$ is defined by 
\begin{align}\label{optimal transport cost}
    C(\mu,\nu)=\inf_{\pi \in \Pi(\mu,\nu)}\int_{\mathcal{X}\times\mathcal{Y}} c(x,y) \pi(dx,dy),
\end{align}
where $c(x,y)$ is the cost for transporting one unit of mass from $x$ to $y$, $\Pi(\mu,\nu)$ is the set of all couplings between $\mu$ and $\nu$, i.e. the set of all probability measures on $\mathcal{X}\times\mathcal{Y}$ with margins $\mu$ and $\nu$. Moreover, when the two Polish probability spaces coincide $\mathcal{X}=\mathcal{Y}$, and the cost is defined in terms of a distance $d$ on $\mathcal{X}$, then one can prove that \eqref{optimal transport cost} actually defines a distance, see e.g. \cite{villani_topics}. 

In fact, let $\mu,\nu$ be two probability measures on a Polish space $(\mathcal{X},d)$, then the Wasserstein distance $ W_{p}$  of order $p\in [1,\infty)$ is defined by the following formula:
\begin{align}\label{Wm}
    (W_{p}(\mu,\nu))^{p}=\inf_{\pi \in \Pi(\mu,\nu)} \int_{\mathcal{X}\times\mathcal{X}} d(x,y)^{p} \pi(dx, dy).
\end{align}
For two $\mathcal{X}$-valued random variables $X$ and $Y$ with $X\sim\mu$ and $Y\sim\nu$, we set $W_p(X,Y)=W_p(\mu,\nu)$.
The Wasserstein space of order $p$ is the space of probability measures which have a finite moment of order $p$ and is defined as 
\begin{align*}
     P_{p}(\mathcal{X}):=\{\mu \in P(\mathcal{X}); \int_{\mathcal{X}} d(x_{0},x)^{p} \mu(dx)<\infty\},
\end{align*}
where $x_{0} \in \mathcal{X}$ is arbitrary and $P(\mathcal{X})$ is the set of all probability measures on $\mathcal{X}$. It turns out that for any $p\in[1,\infty)$, $W_p$ defines a metric on $P_p(\cX)$ \cite[Theorem 7.3]{villani_topics}. Furthermore, the Wasserstein space $P_{p}(\mathcal{X}$) metrized by the Wasserstein distance $W_p$ is also a Polish space  \cite[Theorem 6.18]{Villani}.

Hence, the Wasserstein distance provides a meaningful distance between distributions.  Furthermore, it is related to the notion of weak convergence of measures (see e.g. \cite[Section 6]{Villani}) and it has various applications in stochastic analysis, especially in ergodicity theory, see for example 
\cite{Friesen,MR4241464,MR4153590,arxiv.2301.05640,arxiv.2301.05120} for applications.
Remark that the following theorem guarantees that the infimum  in \eqref{Wm} is reached by some optimal coupling.

\begin{Theorem}[{Existence of optimal coupling, \cite[Theorem 4.1]{Villani}}]\label{Existence of optimal coupling}
Let $(\cX,d)$ be a Polish space and $\mu,\nu \in P_{p}(\cX)$. Then there exists an optimal coupling $\hat{\pi} \in \Pi(\mu,\nu)$ such that
\begin{align}\label{infC}
   (W_{p}(\mu,\nu))^{p}&= \int_{\cX\times\cX} d(x,y)^{p} \hat{\pi}(dx, dy).
\end{align}
\end{Theorem}
We should hence expect that at least in some cases, this infimum can be identified by writing the couplings in terms of ``copulas''. This seems to us natural, as the Wasserstein distance itself is defined by minimizing over all couplings between two marginal distributions, while the copula identifies the exact dependence between the marginal distributions of a coupling.

Abe Sklar introduces the word 'copula' in the statistical literature in 1959 \cite{MR0125600,MR1485519}. to define functions capable of linking a multidimensional distribution to its marginal distributions. Functions with these properties had already been used before by authors such as Giorgio Dall'Aglio \cite{DALL’AGLIO} and Maurice Frechèt \cite{Frechet}, without giving them a name. Giorgio Dall’Aglio \cite[Theorem X]{DALL’AGLIO} was indirectly the first to link that the $p$-Wasserstein distance for $d=1$ could be written in terms of the comonotonicity copula $M$, however  without being aware of the concept of copulas or Wasserstein distance. This Theorem  \cite[Theorem X]{DALL’AGLIO} of Dall'Aglio is reported here in Theorem \ref{tails} in section \ref{sec:main}.

Nowadays, copulas are an important tools in actuarial science, since these functions are able through the marginals to clarify the dependence structure of risks \cite{denuit2005actuarial,MR1968943,MR2244349,MR3445371}.

 In \cite{vallander_ru}, Vallander established the following link between the $1$-Wasserstein distance with the Euclidean distance and the generalized inverse   of one-dimensional distribution functions. 
\begin{Theorem}\label{ThVallender}
    Let $\mu, \nu$ be two probability measures in $(\R,\mathcal{B}(\R))$.
Let $F$ and $G$ be the associated distribution functions.
   \begin{align}
    W_{1}(\mu,\nu) &= \int_{[0,1]} |F^{-1}(u)-G^{-1}(u)| du, \end{align} 
where $F^{-1}$ and $G^{-1}$ are the generalized inverses associated to $F$ and $G$, respectively.
\end{Theorem}

The aim of this article is to  write the $p$-Wasserstein distance  with the $p$-norm in terms of copula. In particular for the case of one-dimensional distributions, we obtain in  Theorem  \ref{main theorem}   that the copula employed to get the optimal coupling of the Wasserstein distances in Theorem \ref{Existence of optimal coupling} is the comotonicity copula. This is the maximizer for all copulas according to the Theorem of Fr\'echet-Hoeffding, recalled in Section \ref{sec:prelim}. The result obtained by Vallander in  Theorem \ref{ThVallender}  is then reformulated as a Corollary \ref{CorVallender} which follows directly from  our result in  Theorem \ref{main theorem}  and Corollary  \ref{comonotonicity result}. 
In Theorem \ref{thm:to_be_deleted} we obtain the equivalent result of Theorem  \ref{main theorem} also for $d$-dimensional distributions  under the  hypothesis that these share the same copula. In other words we have to make the assumption that in both $d$-dimensional distributions the one-dimensional marginals have the same dependence structure (see Definition \ref{share same copula} below). As explained in Remark \ref{RemarkNecessity} this is not only a sufficient but also a necessary hypothesis. The result of Vallander in Theorem \ref{ThVallender} was generalized by Alfonsi and Jourdain  for the $d$-dimensional case \cite{Alfonsi} again by  assuming that the $d$-dimensional distributions share the same copula and is reformulated here as Corollary \ref{The Wasserstain distance in terms of the generalized inverses} of our  Theorem \ref{thm:to_be_deleted}.   (See also \cite{BDS} for an infinite dimensional version of Vallander's Theorem \ref{ThVallender}  involving Banach space valued distributions)\\
Finally, let us consider a more general definition of the Wasserstein distance, denoted in \cite{Alfonsi} by $W_{p,q}$, and defined through
\begin{equation}\label{generaldefwasserstein}
(W_{p,q}(\mu,\nu))^{p}= \inf_{\pi \in \Pi(\mu, \nu)} 
    \int_{\R^{d} \times \R^{d}}  \|x-y\|_q^p \pi(dx,dy).    
\end{equation}
Assuming $p\neq q$, $p,q \geq 1$,  and that the probability measures $\mu$ and $\nu$ are sharing the same copula, Alfonsi \cite[Proposition 1.1]{Alfonsi} proved that $W_{p,q}$ cannot always be decomposed in terms of the generalized inverse functions.  However, since all norms are equivalent on $\R^d$, we get an upper and lower bounds of the Wasserstein distance $W_{p,q}$ in terms of the comonotonicity copula in Proposition \ref{bound} and consequently also in terms of the generalized inverse functions. 

The article is structured as follows. 
In Section \ref{sec:prelim}, we give a basic introduction into the fundamental results of copula theory. Section \ref{sec:main}  is devoted to the statements and  proofs of the main results Theorem \ref{main theorem} and Theorem  \ref{thm:to_be_deleted}, where the Wasserstein distance $W_p$ is obtained in terms of the comonotonicity copula. Moreover, for the case of the Wasserstein distance $W_{p,q}$ we obtain a lower and upper bounds in terms of $W_p$  in Proposition \ref{bound}. In Corollary \ref{CorhalfAlfonsi} we get as a consequence the lower and upper bounds of $W_{p,q}$ can be written in terms of  generalized inverse functions.
\section{Preliminaries}\label{sec:prelim}
In order to prove the main theorems of this article, we set up the notion of generalized inverse functions and recall the fundamentals of copula theory. Note that both univariate and multivariate (or joint) distribution functions are understood in the probabilistic sense, i.e., we also assume right continuity. \\
Firstly, it is important to recall the definition of the generalized inverse function \cite{embrechts2013note}.
\begin{Definition}
    Let $F:\R\rightarrow[0,1]$ be a distribution function. The generalized inverse function or quantile function of $F$ is defined by
    \begin{align*}
         F^{-1}(u)=\inf\{x \in \R \colon F(x)\geq u\}, u \in [0,1].
    \end{align*}
\end{Definition}
For the definition of the Wasserstein distances, we need the notion of coupling. Recall that for probability measures $\mu,\nu$ on $\cB(\R^d)$, a probability measure $\pi$ on $\cB(\R^d\times\R^d)$ is called \emph{coupling} of $\mu$ and $\nu$ if for any Borel set $A\in\cB(\R^d)$,
\begin{displaymath}
 \pi(A\times\R^d)=\mu(A)\text{ and }\pi(\R^d\times A)=\nu(A).
\end{displaymath}
In other words, the marginals of $\pi$ are given by $\mu$ and $\nu$.
Next, we introduce the notion of copula, which describes the dependence  of a random vector from its marginals. This dependence is well understood  understood thanks to   Sklar's Theorem \ref{Sklar's theorem}, cited below.
\begin{Definition}
A function $C\colon[0,1]^d\to[0,1]$ is called a ($d$-dimensional) copula if the following properties are fulfilled:
\begin{enumerate}
    \item For all $u\in[0,1]^d$ s.t. $u_i=0$ for some $i$, we have $C(u)=0$ (groundedness).
    \item For all $(a_1,\dotsc,a_d),(b_1,\dotsc,b_d)\in[0,1]^d$ with $a_i\leq b_i$ for all $i$, we have
    \begin{displaymath}
     \sum_{i_1=1}^2\dotsi\sum_{i_d=1}^2(-1)^{i_1+\dotsb i_d}C(u_{1i_1},\dotsc,u_{d i_d})\geq 0,
    \end{displaymath}
    where $u_{j1}=a_j$ and $u_{j2}=b_j$ for all $j\in\{1,\dotsc,d\}$ ($d$-increasing).
    \item $C(1,\dotsc,1,u_i,1,\dotsc,1)=u_i$ for all $i\in\{1,\dotsc,d\}$ and $u_i\in[0,1]$ (uniform margins).
\end{enumerate}
\end{Definition}
By inspecting the definition of copulas, one notices that they can be viewed as multivariate distribution functions on $[0,1]^d$ which have uniform margins on $[0,1]$. 
A detailed introduction and analysis of copulas can be found in \cite{Nelsen}. We review some of the essential results below. The following theorem is fundamental in the theory of copulas. It shows the relationship between multivariate distribution functions and their univariate margins.
\begin{Theorem}[{Sklar's theorem, \cite[Theorem 2.10.9]{Nelsen}}]\label{Sklar's theorem}
Let $H$ be a $d$-dimensional distribution function with margins $F_{1},\dotsc, F_{d}$. Then there exists a $d$-copula $C$ such that for all $x=(x_{1},\dotsc,x_{d}) \in\R^{d}$
\begin{equation}\label{eq:sklar_identity}
    H(x_{1},\dotsc,x_{d})=C(F_{1}(x_{1}),\dotsc,F_{d}(x_{d})).
\end{equation}
On the other hand, let $C$ be a copula and $F_1,\dotsc,F_d$ one-dimensional distribution functions. Then the function $H$ defined by \eqref{eq:sklar_identity} is a $d$-dimensional joint distribution function with margins $F_1,\dotsc,F_d$.
\end{Theorem}
One elementary result regarding joint distribution functions and copulas are the so-called Fr\'echet-Hoeffding bounds. They are defined as 
  \begin{equation}\label{Comonotinicity copula}
    M^{d}(u_1,\dotsc,u_d) := \min(u_1,\dotsc,u_d) 
  \end{equation}
   \begin{equation*}
    W^{d}(u_1,\dotsc,u_d) := \max(u_1+u_2+\dotsb+u_d-d+1,0)
  \end{equation*}
In the sequel, we omit the dimension index when no confusion may arise.
While $M^d$ is a $d$-copula for any $d\geq 2$, known as the comonotonicity copula, this is not true for $W^d$ as soon as $d>2$. However, these functions do not only represent essential dependence structures, they also serve as elementary bounds for any other copula. More precisely, we have the following result:

\begin{Theorem}[{Fréchet-Hoeffding bounds in $d$-dimensions, \cite[Theorem. 2.10.12]{Nelsen}}]\label{Fréchet-Hoeffding bounds}
    For any $d$-copula $C$ and for all $(x_1,\dotsc, x_d) \in \R^d$, 
    \begin{align*}
    W^{d}(F_{1}(x_{1}),\dotsc,F_{d}(x_{d})) &\leq C(F_{1}(x_{1}),\dotsc,F_{d}(x_{d})) \leq  M^{d}(F_{1}(x_{1}),\dotsc,F_{d}(x_{d})),
    \end{align*}
\end{Theorem} 
The following theorem makes this notion precise.
heFor $a<b$, denote by $\cU(a,b)$ the uniform distribution on the interval $(a,b)$.
\begin{Theorem}[{Equivalent conditions comonotonicity, \cite[Theorem 2]{DDGKV2001}}]\label{Equivalent conditions comonotonicity}
Let $(\Omega, \mathcal{F}, \P)$ be a probability space, $X,Y$ be two $\R$-valued random variables on $(\Omega, \mathcal{F}, \P)$ with distribution functions $F$ and $G$, respectively and joint distribution function $H$. A random vector $(X,Y)$ is comonotonic if and only if one of the following equivalent conditions holds:
    \begin{enumerate}
     \item For all $(x,y) \in \R^{2}$, we have
        \begin{align*}
            H(x,y)= M(F(x),G(y)).
        \end{align*}
        \item For $U \sim \cU(0,1)$, we have
        \begin{align} \label{comonotonic vector}
            (X,Y)\overset{d}{=} (F^{-1}(U), G^{-1}(U)).
        \end{align}
    \end{enumerate}
\end{Theorem}
 A random vector $(X,Y)$ is called  comonotonic if it satisfies \eqref{comonotonic vector} in Theorem \ref{Equivalent conditions comonotonicity}.
Hence, the following corollary, that is used in the proofs of our main results Theorem \ref{main theorem} and Theorem \ref{thm:to_be_deleted}, is an immediate consequence.

\begin{Corollary}\label{comonotonicity result}
   Let $(\Omega, \mathcal{F}, \P)$ be a probability space, $X,Y$ be two $\R$-valued random variables on $(\Omega, \mathcal{F}, \P)$ with distribution functions $F$ and $G$. If $(X,Y)$ is a comonotonic random vector then for any integrable measurable function $g$,
 \begin{equation*}
  \E[g(X,Y)]=\int_{[0,1]} g(F^{-1}(u),G^{-1}(u)) du.
 \end{equation*} 
\end{Corollary}

\section{Main results and discussion}\label{sec:main}

Our first result is the representation of the Wasserstein distance on $\R$ in terms of the comonotonicity copula. In fact, we show that the optimal coupling is always given by the two-dimensional comonotonicity copula, called here also $M$-copula.
\begin{Theorem}[Wasserstein distance in terms of copula in $\R$]\label{main theorem}
Let $\mu, \nu$ be two probability measures in $P_{p}(\R)$. 
Let $F$ and $G$ be the associated distribution functions. Then for all $p\in [1,\infty)$,
\begin{align*}
    (W_{p}(\mu,\nu))^{p}&=\int_{\R}\int_{\R} |x-y|^{p} dM(F(x),G(y))
\end{align*}
\end{Theorem}
The proof of Theorem \ref{main theorem} will be divided into the cases $p=1$ and $p>1$.
\begin{proof}[Proof of Theorem \ref{main theorem}, $p=1$].
By Theorem \ref{optimal transport cost}, there exists an optimal coupling $\hat{\pi}\in\Pi(\mu,\nu)$ such that
\begin{equation}\label{optimal coupling}
    W_1(\mu,\nu)=\int_{\R}\int_{\R}\lvert x-y\rvert\hat{\pi}(dx,dy).
\end{equation}
Hence, following \cite[Theorem 14.1]{Billingsley}, for $F$ and $G$ two distribution functions there exists on some probability space $(\Omega, \mathcal{F}, \P)$ a couple of two random variables $X$ and $Y$  such that $X\sim \mu$, and $Y \sim \nu$ and whose joint distribution function is $\hat{\pi}$. Hence, from \eqref{optimal coupling} follows
\begin{equation*}
    W_1(\mu,\nu)= \E\big[|X-Y| \big].
\end{equation*}
Moreover, it was proven in \cite{vallander_ru} that 
 \begin{align*}
  \E\big[|X-Y| \big] &= \int_{\R} \left[ \P(X\leq y)+\P(Y\leq y)-2\P(X\leq y, Y\leq y)  \right] dy.
 \end{align*}
Denoting by $C$ the copula corresponding to $X$ and $Y$, by Theorems \ref{Sklar's theorem} and \ref{Fréchet-Hoeffding bounds}, we get
    \begin{align*}
        \E\big[|X-Y| \big]&= \int_{\R} [F(y)+G(y)-2\P(X\leq y, Y\leq y)]dy 
        \\
        &= \int_{\R} [F(y)+G(y)-2C(F(y),G(y))]dy
         \\
         &\geq  \int_{\R} [F(y)+G(y)-2M(F(y),G(y))]dy.
    \end{align*}
Let $(\tilde{X},\tilde{Y})$ be a comonotonic random vector with distributions $\mu$ and $\nu$, respectively. Its copula is given by $M$. Therefore,
\begin{align}
       W_1(\mu,\nu)\ge\int_{\R} [ F(y)+G(y)-2M(F(y), G(y))] dy= \E\big[|\tilde{X}-\tilde{Y}|\big].
   \end{align}
    Furthermore, it follows
    \begin{align*}
        W_{1}(\mu,\nu)= \E\big[|\tilde{X}-\tilde{Y}|\big]&=\int_{\R} \int_{\R} |x-y|^{p} dM(F(x),G(y)).
    \end{align*} 
\end{proof}
In order to consider the case $p>1$, we first introduce the following result by Dall'Aglio \cite{DALL’AGLIO}. 
\begin{Theorem}\cite[Theorem X]{DALL’AGLIO}.\label{tails}
 Let $p\in (1,\infty)$. Let $F$ and $G$ distribution functions satisfying 
  \begin{equation} \label{first assymption DALL’AGLIO }
       \lim_{x\rightarrow +\infty} x^p [1-F(x)]=\lim_{x\rightarrow +\infty} x^pF(-x)=0,
  \end{equation}
  \begin{equation} \label{second assymption DALL’AGLIO }
       \lim_{x\rightarrow +\infty} x^p [1-F(x)]=\lim_{x\rightarrow +\infty} x^pF(-x)=0,
  \end{equation}
Let $\pi$ be a joint distribution function with marginals $F$ and $G$. Then, 
\begin{align*}
    I(\pi):=\int_{\R} \int_{\R} |x-y|^p d\pi(x,y)
\end{align*}
is minimized by
\begin{equation}
    \pi(x,y)=\min\{F(x),G(y)\}.
\end{equation}
\end{Theorem}
Fix $\mu,\nu$ with associated distribution functions $F$ and $G$. By assuming $\eqref{first assymption DALL’AGLIO }$ and $\eqref{second assymption DALL’AGLIO }$ on the tails of the margins $F$ and $G$, Dall'Aglio came to prove that the double integral $I(\pi)$ above can be minimized. However, $I(\pi)$ is the expression to be minimized in the Wasserstein distance, i.e.,
\begin{displaymath}
 (W_p(\mu,\nu))^{p}=\inf_{\pi}\int_{\R}\int_{\R}|y-x|^pd\pi(x,y)
\end{displaymath}
where the infimum runs over all bivariate distribution functions $\pi$ with margins $F$ and $G$. It is sufficient for $\eqref{first assymption DALL’AGLIO }$ and $\eqref{second assymption DALL’AGLIO }$ to hold to have moment assumptions of the corresponding order $p$ (see \cite[Theorem 6.5.2]{Giorgio2003}). Since we are working  on the Wasserstein space of order $p\in [1,\infty)$ the existence of finite moments of order $p$ is guaranteed.
\begin{proof} [Proof of Theorem \ref{main theorem}, $p>1$]
We may, as before, recall from Theorem \ref{optimal transport cost} the existence of an optimal coupling $\hat{\pi}\in\Pi(\mu,\nu)$. However, using Theorem \ref{tails} we have that $(W_p(\mu,\nu))^p$ is minimized by the copula $M$ and hence
\begin{displaymath}
 (W_p(\mu,\nu))^{p}=\int_{\R}\int_{\R}|y-x|^pdM(F(x),G(y)). 
\end{displaymath}
\end{proof}
The above result Theorem \ref{main theorem}, in which the Wasserstein distance is expressed in terms of the copula $M$, allows to obtain again the known  result in terms of the generalized inverse functions, see e.g. \cite[Theorem 2.18 and Remark 2.19]{villani_topics}.

\begin{Corollary} \label{CorVallender}
Let $\mu, \nu$ be two probability measures in $P_{p}(\R)$. 
Let $F$ and $G$ be the associated distribution functions. Then for all $p\in [1,\infty)$,
\begin{equation}\label{equationcor}
    (W_{p}(\mu,\nu))^{p}=\int_{[0,1]} |F^{-1}(u)-G^{-1}(u)|^{p} du, 
\end{equation}
where 
$F^{-1}$ and $G^{-1}$ are the generalized inverses (or the quantile function of $F$ and $G$) associated to $F$ and $G$ respectively. 
\end{Corollary}

\begin{proof}
Let $(\tilde{X},\tilde{Y})$ be a comonotonic random vector with distributions $\mu$ and $\nu$, respectively. From Theorem \ref{main theorem} and Corollary \ref{comonotonicity result} follows,
\begin{align*}
    (W_p(\mu,\nu))^{p}= \E \big[|\tilde{X}-\tilde{Y}|^p \big]&=\int_{[0,1]} |F^{-1}(u)-G^{-1}(u)|^{p} du. 
\end{align*}
\end{proof}

\begin{Remark}
    As Vallander also noticed in \cite{vallander_ru}, for $p=1$, the Wasserstein distance on $\R$ may even be expressed in terms of the distribution functions $F$ and $G$:
     \begin{displaymath}
      W_1(\mu,\nu)=\int_\R|F(x)-G(x)|dx.
     \end{displaymath} 
\end{Remark}
Our aim, now, is to generalize the result obtained in Theorem \ref{main theorem} to the $d$-dimensional case, i.e. writing the expression of the Wasserstein distance for $d$-dimensional distribution using the comonotonicity copula. 
However, to achieve this, we require that  both probability measures share the same dependence structure, i.e., the same copula.
Namely:
\begin{Definition} \label{share same copula}
 Let $\mu,\nu\in P(\R^d)$ be two probability measures and $(F_1,\dotsc,F_d)$ and $(G_1,\dotsc,G_d)$ their one-dimensional marginal distribution functions. We say that $\mu$ and $\nu$ share the same copula if there exists a $d$-dimensional copula $C\colon[0,1]^d\to[0,1]$ such that the joint distribution functions $H_\mu$ and $H_\nu$ of $\mu$ and $\nu$ can be written as follows:
\begin{equation}\label{eq:shared_copula1}
\begin{split}
 H_\mu(x_1,\dotsc,x_d)&=C(F_1(x_1),\dotsc,F_d(x_d))
 \\
 H_\nu(y_1,\dotsc,y_d)&=C(G_1(y_1),\dotsc,G_d(y_d)).
\end{split}
\end{equation}
Similarly, we say that $X\sim\mu$ and $Y\sim\nu$ share a copula $C$ if their distributions $\mu,\nu$ do so. 
\end{Definition}  For such measures with \eqref{eq:shared_copula1}, it has been shown 
\begin{Theorem}[{\cite[Theorem 5]{BDS}, \cite[Theorem 2.9]{MR1240428}}]\label{thm:bds}
Let $X,Y$ be random variables on $\R^d$. Then the following are equivalent:
\begin{enumerate}
    \item $X$ and $Y$ share the same copula $C$.
    \item The Wasserstein distance between $X$ and $Y$ is given by
    \begin{displaymath}
     (W_p(X,Y))^{p}=\sum_{i=1}^d (W_p(X_i,Y_i))^{p}
    \end{displaymath}
\end{enumerate}
\end{Theorem}

As a consequence, through the previous result Theorem \ref{thm:bds} and our formulation in Theorem \ref{main theorem} of the Wasserstein distance in terms of $M$ we obtain the following statement.

\begin{Theorem}\label{thm:to_be_deleted}
Let $\mu,\nu$ be two probability measures in $P_p(\R^d)$ sharing the same copula. Denote by $F_i$ and $G_i$ ($i=1,\dotsc,d$) the distribution functions of the one-dimensional margins of $\mu$ and $\nu$, respectively. Then for all $p\in [1,\infty)$, 
\begin{equation}\label{equationd}
    ( W_{p}(\mu,\nu))^{p}=\sum_{i=1}^{d} \int_{\R} \int_{\R} |x_{i}-y_{i}|^{p} dM(F_{i}(x_{i}),G_{i}(y_{i}))
\end{equation}
where $M$ is the copula defined by \eqref{Comonotinicity copula}.
\end{Theorem}
\begin{proof}
Denoting by $\mu_i$ and $\nu_i$ the one-dimensional margins of $i$-th coordinate of the two probability measures $\mu$ and $\nu$. respectively. Since $\mu$ and $\nu$ share the same copula, applying Theorems \ref{thm:bds} and \ref{main theorem}, we have
\begin{align*}
     (W_{p}(\mu,\nu))^{p}&=\sum_{i=1}^{d} (W_{p}(\mu_{i},\nu_{i}))^{p}
     \\ 
    &= \sum_{i=1}^d\int_\R\int_\R|x_i-y_i|^pdM(F_i(x_i),G_i(y_i)).
\end{align*}
\end{proof}
Additionally, the previous result and the Corollary \ref{CorVallender} allows us to directly express the Wasserstein distance through the generalized inverse functions.
\begin{Corollary}\label{The Wasserstain distance in terms of the generalized inverses}
Let $\mu,\nu$ be two probability measures in $P_p(\R^d)$ sharing the same copula. Denote by $F_i$ and $G_i$ ($i=1,\dotsc,d$) the distribution functions of the one-dimensional margins of $\mu$ and $\nu$, respectively. Then for all $p\in [1,\infty)$,
\begin{equation}\label{wassersteindistanceddimen}
     (W_{p}(\mu,\nu))^{p}= \int_{[0,1]} \left \|F^{-1}(u)-G^{-1}(u)\right \|_p^{p} du
\end{equation}
\begin{equation*}
    =\sum_{i=1}^d\int_{[0,1]}|F_i^{-1}(u)-G_i^{-1}(u)|^pdu,
\end{equation*}
where $M$ is the copula defined by \eqref{Comonotinicity copula}, $F^{-1}, G^{-1}:[0,1] \rightarrow \R^{d}$ are the generalized inverses associated to $F$ and $G$ respectively, where for all $u \in [0,1]$
\begin{align*}
    F^{-1}(u)&:= \left( F_{1}^{-1}(u),\dotsc,F_{d}^{-1}(u) \right) \\ G^{-1}(u)&:=\left( G_{1}^{-1}(u),\dotsc,G_{d}^{-1}(u) \right)
\end{align*}
and $F_{i}^{-1}, G_{i}^{-1}, i=1,\dotsc,d$, are one-dimensional generalized inverses. 
\end{Corollary}

\begin{proof} By \eqref{equationd} and \eqref{equationcor}, we get
\begin{align*}
     (W_{p}(\mu,\nu))^{p}&=\sum_{i=1}^d\int_\R\int_\R|x_i-y_i|^pdM(F_i(x_i),G_i(y_i))
     \\ 
    &= \sum_{i=1}^d\int_{[0,1]}|F_i^{-1}(u)-G_i^{-1}(u)|^pdu.
\end{align*}
\end{proof}

\begin{Remark}\label{RemarkNecessity}
In Theorem \ref{thm:to_be_deleted}, the condition that the two probability measures share the same copula is not only sufficient but necessary. In fact, let $\mu$, $\nu$ two different probability measures on $\R^d$ associate a two copulas $C_1$, $C_2$, and with the same $d$-univariate distribution functions $F_1,\dots,F_d$. Let us assume ad absurdum that the result holds also if $C_1$ and $C_2$ are different, then
\begin{align*}
    (W_{p}(\mu,\nu))^{p}&=\sum_{i=1}^{d} \int_{\R} \int_{\R} |x_{i}-y_{i}|^{p} dM(F_{i}(x_{i}),F_{i}(y_{i}))
    \\ 
    &=(W_{p}(\mu,\mu))^{p}=0.
\end{align*}
\end{Remark}

The result in Corollary \ref{The Wasserstain distance in terms of the generalized inverses} was already obtained in Alfonsi {\cite[Poposition 1.1]{Alfonsi}} but in his proof does not use our result of Theorem \ref{thm:bds}.

Moreover, we are going to use the Theorem \ref{thm:to_be_deleted} to bound the Wasserstein distance $W_{p,q}$ defined in \eqref{generaldefwasserstein} for $(\R^d,\|\cdot\|_q)$.
\begin{Proposition}\label{bound}
     Let $\mu,\nu$ be two probability measures in $P_p(\R^d)$ and $q\geq 1$ with $q\neq p$ sharing the same copula. Denote by $F_i$ and $G_i$ ($i=1,\dotsc,d$) the distribution functions of the one-dimensional margins of $\mu$ and $\nu$, respectively. Then the following holds:
\begin{enumerate}
     \item If $1\leq q\leq p <\infty$, we have
     \begin{displaymath}
       \sum_{i=1}^{d} \int_{\R} \int_{\R} |x_{i}-y_{i}|^{p} dM(F_{i}(x_{i}),G_{i}(y_{i}))\leq (W_{p,q}(\mu,\nu))^{p}\leq d^{\frac{p}{q}-1}\sum_{i=1}^{d} \int_{\R} \int_{\R} |x_{i}-y_{i}|^{p} dM(F_{i}(x_{i}),G_{i}(y_{i})).
     \end{displaymath}
     \item If $1\leq p\leq q <\infty$, we have
     \begin{displaymath}
       d^{\frac{p}{q}-1}\sum_{i=1}^{d} \int_{\R} \int_{\R} |x_{i}-y_{i}|^{p} dM(F_{i}(x_{i}),G_{i}(y_{i}))\leq (W_{p,q}(\mu,\nu))^{p}\leq \sum_{i=1}^{d} \int_{\R} \int_{\R} |x_{i}-y_{i}|^{p} dM(F_{i}(x_{i}),G_{i}(y_{i})).
     \end{displaymath}
 \end{enumerate}
where $M$ is the copula defined by \eqref{Comonotinicity copula}.
 \end{Proposition}

 \begin{proof}
   Since all the norms are equivalent, then there exists $c_{1}, c_{2}>0$ constants such that for all $x,y \in \R^{d}$
    \begin{displaymath}
        c_{1}\|x-y\|_{p}\leq \|x-y\|_{q} \leq c_{2} \|x-y\|_{p}. 
    \end{displaymath}
    In fact, if $q\leq p$, we may take $c_{1}=1$ and $c_{2}=d^{\frac{1}{q}-\frac{1}{p}}$. We then get
    \begin{displaymath}
      (W_{p}(\mu,\nu))^{p} \leq  (W_{p,q}(\mu,\nu) )^{p}\leq d^{\frac{p}{q}-1} (W_{p}(\mu,\nu))^{p}
    \end{displaymath}
    A similar calculation shows the case $p\leq q$. Thus, from Theorem \ref{thm:to_be_deleted}  we get the desired result.
   \end{proof}

   As a Corollary, we get
\begin{Corollary} \label{CorhalfAlfonsi}
    Let $\mu,\nu$ be two probability measures in $P_p(\R^d)$ and $q\geq 1$ with $q\neq p$ sharing the same copula. Denote by $F_i$ and $G_i$ ($i=1,\dotsc,d$) the distribution functions of the one-dimensional margins of $\mu$ and $\nu$, respectively. Denote by $F^{-1}, G^{-1}:[0,1] \rightarrow \R^{d}$ are the generalized inverses associated to $F$ and $G$ respectively, where for all $u \in [0,1]$
\begin{align*}
    F^{-1}(u)&:= \left( F_{1}^{-1}(u),\dotsc,F_{d}^{-1}(u) \right) \\ G^{-1}(u)&:=\left( G_{1}^{-1}(u),\dotsc,G_{d}^{-1}(u) \right)
\end{align*}
and $F_{i}^{-1}, G_{i}^{-1}, i=1,\dotsc,d$, are one-dimensional generalized inverses.  
 Then the following holds:
 \begin{enumerate}
     \item If $1\leq q\leq p <\infty$, we have
     \begin{displaymath}
       \int_{[0,1]}\|F^{-1}(u)-G^{-1}(u)\|_p^p du\leq (W_{p,q}(\mu,\nu))^{p}\leq d^{\frac{p}{q}-1}\int_{[0,1]}\|F^{-1}(u)-G^{-1}(u)\|_p^p du.
     \end{displaymath}
     \item If $1\leq p\leq q <\infty$, we have
     \begin{displaymath}
       d^{\frac{p}{q}-1}\int_{[0,1]}\|F^{-1}(u)-G^{-1}(u)\|_p^p du\leq (W_{p,q}(\mu,\nu))^{p}\leq \int_{[0,1]}\|F^{-1}(u)-G^{-1}(u)\|_p^p du.
     \end{displaymath}
 \end{enumerate} 
\end{Corollary}
\begin{proof}
Let $q\le p$, then the statement directly follows by Proposition \ref{bound} and \eqref{wassersteindistanceddimen}. The proof is analogous for $p\le q$.
\end{proof}
{\bf Acknowledgments.} We thank Dennis Schroers (University Bonn) for many  useful discussions related to this article and to point us Remark \ref{RemarkNecessity}. We thank also Stefano Bonaccorsi (University Trento) for giving to our attention the related references \cite{DALL’AGLIO} at the beginning of this work and for many important comments. 

\begin{thebibliography}{9}
\bibitem{Alfonsi}
 Alfonsi, A. and Jourdain, B.,
\textit{A remark on the optimal transport between two probability measures sharing the same copula},
Statist. Probab. Lett.,
vol. 84, pp. 131--134, 2014.
  ISSN: 0167-7152.
  MR Class: 60E05 (60E15).
 MR Number: 3131266.
 DOI: \href{https://doi.org/10.1016/j.spl.2013.09.035}{10.1016/j.spl.2013.09.035}.

 \bibitem{BDS}
Benth, Fred Espen and Di Nunno, Giulia and Schroers, Dennis,
\textit{Copula measures and {S}klar's theorem in arbitrary dimensions},
Scand. J. Stat., vol. 49, 2022, pp. 1144--1183.
ISSN: 0303-6898.
MR Class: 60G07 (62H05).
MR Number: 4471282.
DOI: \href{https://doi.org/10.1111/sjos.12559}{10.1111/sjos.12559}. 
URL: \url{https://doi.org/10.1111/sjos.12559}.

\bibitem{Billingsley}
 Billingsley, Patrick,
  \textit{Probability and Measure},
  Anniversary ed., vol. 338,
  John Wiley \& Sons, Inc., Hoboken, New Jersey, 1995.
  638 pages
  ISBN = {0-471-00710-2, 978-1-118-12237-2}.

\bibitem{MR1240428}
Cuesta-Albertos, J. A. and R\"{u}schendorf, L. and Tuero-D\'{\i}az, A.,
\textit{Optimal coupling of multivariate distributions and stochastic processes},
 J. Multivariate Anal., vol. 46, pp. 335--361, 1993.
ISSN = 0047-259X,
MR Class: 60E99 (62H05).
MR Number: 1240428.
MR Reviewer: Giorgio Dall'Aglio.
DOI: \href{https://doi.org/10.1006/jmva.1993.1064}{10.1006/jmva.1993.1064}. 
URL: \url{https://doi.org/10.1006/jmva.1993.1064}.

\bibitem{DALL’AGLIO}
 Dall'Aglio, Giorgio,
  \textit{Sugli estremi dei momenti delle funzioni di ripartizione doppia},
 Ann. Scuola Norm. Sup. Pisa Cl. Sci. (3),
vol. 10, pp. 35--74, 1956.

\bibitem{Giorgio2003}
Dall'Aglio, Giorgio,
\textit{CALCOLO DELLA PROBABILIT\`A},
Zanichelli editore S.p.A., 2003.
vii+328.

\bibitem{denuit2005actuarial}
Denuit, Michel and Dhaene, Jan and Goovaerts, Marc J. and Kaas, R.,
\textit{Actuarial theory for dependent risks: measures, orders and models},
John Wiley \& Sons, Ltd, 2005.
DOI: \href{https://doi.org/10.1002/0470016450}{10.1002/0470016450}.

\bibitem{DDGKV2001}
 Dhaene, J. and Denuit, M. and Goovaerts, M. J. and Kaas, R. and Vyncke, D.
  \textit{The concept of comonotonicity in actuarial science and finance: theory},
  Insurance Math. Econom.,
  vol. 31, pp. 3--33, 2002.
  ISSN: 0167-6687.
  MR Class: 60E05 (91B30).
 MR Number: 1956509.
 DOI: \href{https://doi.org/10.1016/S0167-6687(02)00134-8}{10.1016/S0167-6687(02)00134-8}.

 \bibitem{embrechts2013note}
Embrechts, Paul and Hofert, Marius,
\textit{A note on generalized inverses},
 Mathematical Methods of Operations Research, vol. 77, pp. 423--432, 2013.
 Springer.

\bibitem{MR1968943}
Embrechts, Paul and H\"{o}ing, Andrea and Juri, Alessandro,
  \textit{Using copulae to bound the value-at-risk for functions of dependent risks},
  Finance Stoch.,
  vol. 7, pp. 145--167, 2003.
 DOI: \href{https://doi.org/10.1007/s007800200085}{10.1007/s007800200085}.

 \bibitem{MR2244349}
Embrechts, Paul and Puccetti, Giovanni,
  \textit{Bounds for functions of dependent risks},
  Finance Stoch.,
  vol. 10, pp. 341--352, 2006.
 DOI: \href{https://doi.org/10.1007/s00780-006-0005-5}{10.1007/s007800200085}.

 \bibitem{MR4241464}
 Farkas, B\'{a}lint and Friesen, Martin and R\"{u}diger, Barbara and Schroers, Dennis,
\textit{On a class of stochastic partial differential equations with multiple invariant measures},
 NoDEA Nonlinear Differential Equations Appl.,
vol. 28, pp. 28--46, 2021.
 DOI: \href{https://doi.org/10.1007/s00030-021-00691-x}{10.1007/s00030-021-00691-x}.

 \bibitem{Frechet}
Fr\'{e}chet, M.,
\textit{Sur les tableaux dont les marges et des bornes sont donn\'{e}es},
 Rev. Inst. Internat. Statist., vol. 28, pp. 10--32, 1960.

 \bibitem{MR4153590}
 Friesen, Martin and Jin, Peng and Kremer, Jonas and R\"{u}diger, Barbara,
\textit{Ergodicity of affine processes on the cone of symmetric positive semidefinite matrices},
 Adv. in Appl. Probab.,
vol. 52, pp. 825--854, 2020.
 DOI: \href{https://doi.org/10.1007/s00030-021-00691-x}{10.1007/s00030-021-00691-x}.

 \bibitem{Friesen}
 Friesen, Martin and Jin, Peng and R\"{u}diger, Barbara,
  \textit{Stochastic equation and exponential ergodicity in {W}asserstein distances for affine processes},
  Ann. Appl. Probab.,
  vol. 30, pp. 2165--2195, 2020.
  DOI: \href{https://doi.org/10.1214/19-AAP1554}{10.1214/19-AAP1554}.

  \bibitem{MR129016}
 Kantorovich, L. V.,
\textit{Mathematical methods of organizing and planning production},
Management Sci., vol. 6, 1959/60, pp.366--422.
English translation of the 1939 article.
DOI: \href{https://doi.org/10.1287/mnsc.6.4.366}{10.1287/mnsc.6.4.366}. 

\bibitem{arxiv.2301.05640}
Kuchling, Peter and Rüdiger, Barbara and Ugurcan, Baris,
\textit{Stability properties of some port-{H}amiltonian {SPDE}s},
arXiv preprint 2301.05640, 2023.
Publisher: arXiv.
DOI: \href{https://doi.org/10.48550/ARXIV.2301.05640}{10.48550/ARXIV.2301.05640}.
URL: \url{https://arxiv.org/abs/2301.05640}.

\bibitem{arxiv.2301.05120}
Mandrekar, Vidyadhar and Rüdiger, Barbara,
\textit{Stability properties of mild solutions of {SPDE}s related to pseudo differential equations},
arXiv preprint, 2301.05120, 2023.
Publisher: arXiv.
DOI: \href{https://doi.org/10.48550/ARXIV.2301.05120}{10.48550/ARXIV.2301.05120}.
URL: \url{https://arxiv.org/abs/2301.05120}.

\bibitem{MR3445371}
McNeil, Alexander J. and Frey, R\"{u}diger and Embrechts, Paul,
\textit{Quantitative risk management}, Revised ed.,
Princeton Series in Finance.
Concepts, techniques and tools.
Princeton University Press, Princeton, NJ, 2015.
xix+699 pp.
ISBN: 978-0-691-16627-8.
MR Class: 91-02 (62G32 62M10 62P05 91B30 91Gxx).
MR Number: 3445371.

\bibitem{Nelsen}
Nelsen, Roger B.,
  \textit{An introduction to copulas},
  Springer Series in Statistics.
  Springer, New York, 2006
  xiv+269 pp.
 DOI: \href{https://doi.org/10.1007/s11229-005-3715-x}{10.1007/s11229-005-3715-x}.
\bibitem{MR1485519}
Sklar, A.,
\textit{Random variables, distribution functions, and copulas---a personal look backward and forward},
In \textit{Distributions with fixed marginals and related topics (Seattle, WA, 1993)},
IMS Lecture Notes Monogr. Ser., vol. 28, pp. 1--14,
Inst. Math. Statist., Hayward, CA, 1996.
ISBN: 0-940600-40-4.
MR Class: 60-03 (01A60 60E05 62H05).
MR Number: 1485519.
DOI: \href{https://doi.org/10.1214/lnms/1215452606}{10.1214/lnms/1215452606}.
URL: \url{https://doi.org/10.1214/lnms/1215452606}.

\bibitem{MR0125600}
Sklar, M.,
\textit{Fonctions de r\'epartition \`a{} {$n$} dimensions et leurs marges},
 Publ. Inst. Statist. Univ. Paris., vol. 8, pp. 229--231, 1959.
ISSN = 0047-259X,
MR Class:60.20.
MR Number: 125600.
MR Reviewer: M.\ Lo\`eve.


\bibitem{vallander_ru}
Vallender, S. S.,
\textit{Calculations of the {V}asser\v{s}te\u{\i}n distance between probability distributions on the line},
Teor. Verojatnost. i Primenen., vol. 18, 1973, pp. 824--827.
ISSN: 0040-361x.
MR Class: 60B05.
MR Number: 0328982.
MR Reviewer: I. Csiszar.

\bibitem{villani_topics}
 Villani, C\'{e}dric,
  \textit{Topics in optimal transportation},
  vol. 58 of Graduate Studies in Mathematics
  American Mathematical Society, Providence, RI, 2003
  xvi+370 pages.
  ISBN = {0-8218-3312-X},
  MR Class = {90-02 (28D05 35B65 35J60 49N90 49Q20 90B20)},
  MR Number = {1964483},
  DOI: \href{https://doi.org/10.1090/gsm/058}{10.1090/gsm/058}.
  URL: \url{https://epubs.siam.org/doi/10.1137/1118101}.
  
  \bibitem{Villani}
Villani, C\'{e}dric,
\textit{Optimal transport old and new},
2nd ed., vol. 338 of Grundlehren der mathematischen Wissenschaften [A Series of Comprehensive Studies in Mathematics],
Springer-Verlag Berlin Heidelberg, 2009
xvii+971 pages.
ISBN: 978-3-540-71049-3.





  










\end{thebibliography}

\end{document}